\documentclass[12pt,a4paper]{article}
\usepackage{amsfonts}
\usepackage{amsthm}
\usepackage{amsmath}
\usepackage{amscd}
\usepackage[latin2]{inputenc}
\usepackage{t1enc}
\usepackage{enumerate}
\usepackage[mathscr]{eucal}
\usepackage{indentfirst}
\usepackage{graphicx}
\usepackage{graphics}
\usepackage{pict2e}
\usepackage{epic}
\numberwithin{equation}{section}
\usepackage[a4paper,top=2cm,bottom=2cm,left=1.5cm,right=1.5cm]{geometry}
\usepackage{epstopdf} 
\usepackage{hyperref}
\usepackage{todonotes}
\usepackage{xcolor}
\usepackage{cancel}
\usepackage{pgfplots}
\pgfplotsset{compat=1.15}
\usepackage{mathrsfs}
\usetikzlibrary{arrows}

\DeclareMathOperator{\Tr}{Tr}

\newcommand{\cB}{{\mathcal B}}

\newcommand{\cM}{{\mathcal M}}

\newcommand{\cL}{{\mathcal L}}

\newcommand{\F}{{\mathbb F}}
\newcommand{\Fq}{\F_q}

\newcommand{\la}{\langle}
\newcommand{\ra}{\rangle}

\newcommand{\FG}{\PG(2,p^n)}
\newcommand{\FGq}{\PG(2,q^h)}

\newcommand{\Fqh}{\mathbb F_{q^h}}

\DeclareMathOperator{\PG}{{PG}}

\theoremstyle{plain}
\newtheorem{Th}{Theorem}[section]
\newtheorem{Lemma}[Th]{Lemma}
\newtheorem{Cor}[Th]{Corollary}
\newtheorem{Prop}[Th]{Proposition}

\theoremstyle{definition}
\newtheorem{Def}[Th]{Definition}
\newtheorem{Conj}[Th]{Conjecture}
\newtheorem{Rem}[Th]{Remark}
\newtheorem{?}[Th]{Problem}
\newtheorem{Ex}[Th]{Example}

\begin{document}
	
	\title{Small 3-fold blocking sets in $\mathrm{PG}(2,p^n)$}
	
	\author{B. Csajb\'ok \thanks{The first author was supported by the J\'anos Bolyai Research Scholarship of the Hungarian Academy of Sciences and partially by the ELTE TKP 2021-NKTA-62 funding scheme and by the   
			National Research, Development and Innovation Fund -- grant numbers ADVANCED 153080 and  EXCELLENCE 151504.}, M.\ R. Kepes\thanks{The second author was supported by the National Research, Development and Innovation Fund -- grant number ADVANCED 153080.}, E. M. Robin, B. S\'ogor, S. Wang, E. Williams}

	\maketitle
	

	\begin{abstract}
		A \(t\)-fold blocking set of the finite Desarguesian plane $\mathrm{PG}(2,p^n)$, $p$ prime, is a set of points meeting each line of the plane in at least \(t\) points. The minimum size of such sets is of interest for numerous reasons; however, even the minimum size of nontrivial blocking sets (i.e. $1$-fold blocking sets not containing a line) in \(\mathrm{PG}(2,p^n)\) is an open question when $n\geq 5$ is odd. 
		For $n>1$ the conjectured lower bound for this size is $(p^n+p^{n(s-1)/s}+1)$, where $p^{n/s}$ is the size of the largest proper subfield of $\mathbb{F}_{p^n}$. Since the union of $t$ pairwise disjoint nontrivial blocking sets is a $t$-fold blocking set, it is conjectured that when $p^{n/s}$ is large enough w.r.t. $t$, then the minimum size of a \(t\)-fold blocking set in \(\mathrm{PG}(2,p^n)\) is \(t(p^n+p^{n(s-1)/s}+1)\).
		If $n$ is even, then the decomposition of the plane into disjoint Baer subplanes gives a $t$-fold blocking set of this size. However, for odd $n$, the existence of such sets is an unsolved problem in most cases.
		
		In this paper, we construct $3$-fold blocking sets of conjectured size. These blocking sets are obtained as the disjoint union of three linear blocking sets of R\'edei type, and they lie on the same orbit of the projectivity $(x:y:z)\mapsto (z:x:y)$.
	\end{abstract}

	\section{Introduction}
	
	A set of points \(\cB\) in a finite projective plane is a \emph{blocking set} if \(\cB\) meets each line of the plane in at least one point. It is easy to see that the smallest blocking sets are lines. The question naturally arises: can we determine the minimum size of \emph{nontrivial} blocking sets (blocking sets that do not contain a line)? This has proven to be a difficult problem and has been widely studied in finite geometry. Similarly, we can define a \emph{\(t\)-fold blocking set} to be a point set that has at least \(t\) points on each line of the finite projective plane. One may construct such sets by taking the union of \(t\) disjoint nontrivial blocking sets. In this paper, we mainly focus on multiple blocking sets of \(\PG(2,q)\), the projective plane over the finite field of order \(q\). The study of small \(t\)-fold blocking sets has many applications, such as the theory of linear error correcting codes \cite{JWP}, or colorings of projective planes (see Section \ref{subsec:chrom}). Although the minimum size of such sets is of interest, it is unknown in most cases, which leads us to consider the following conjecture, motivated by conjectures of Sziklai from \cite{linconj} and by our Proposition \ref{PoSt2} regarding the exponent of blocking sets:
	\begin{Conj}
		\label{conj0}
		If $q=p^n$, $n>1$, and the size $p^{n/s}$ of the largest subfield of $\F_{p^n}$ is large enough w.r.t. $t$, then the smallest possible size of a $t$-fold blocking set of $\PG(2,p^n)$ is 
		\begin{equation}
			\label{bound}
			t(p^n+p^{n(s-1)/s}+1).
		\end{equation}
	\end{Conj}
	
	In Sziklai's paper, it is conjectured that small blocking sets are linear (see Section \ref{Sec:2.2} for the definitions), and when multiplicities for the points are also allowed, small $t$-fold blocking sets are the union of some (not necessarily disjoint) linear multiple blocking sets. In the same paper, the bound (\ref{bound}), with $t=1$, is conjectured as a lower bound for the size of certain linear blocking sets. 
	This latter conjecture was proved in \cite{dbvdv} under an extra condition. 
	
	We do not deal with the case when \(n=1\), i.e. projective planes over a field of prime order, as in that case blocking sets behave quite differently. In $\PG(2,2)$ all blocking sets are trivial, for $p>2$ prime Blokhuis showed in \cite{Blokhuis0} that the smallest  blocking sets in \(\PG(2,p)\) are of size \(3(p+1)/2\). 
	In \cite{Ball0} for \(p>3\) prime and \(t\leq (p+1)/2\) Ball proved the lower bound \(tp+t+(p+1)/2\) for the size of $t$-fold blocking sets.
	Even the construction of $2$-fold blocking sets of size smaller than \(3p\) (the size of the trivial double blocking set: a triangle) has proved to be a difficult task. Braun, Kohnert and Wassermann showed the existence of such sets for \(p=13\) (see \cite{arcs});  Csajb\'ok and H\'eger presented constructions of size \(3p-1\) for \(p=13,19,31,37,43\) (see \cite{3q-1double}).    
	
	We briefly recall what is known in the literature regarding Conjecture \ref{conj0}. 
	Note that $s$ is the smallest prime divisor of $n$.
	
	\begin{enumerate}
		\item For $t=1$, Example \ref{traceex} reaches the bound, cf. \cite[pg.\ 138]{Bsurvey}, or the equivalent construction of T.G. Ostrom from \cite[Theorem 5.3]{Bruen1}.
		\item For $t=1$, $n=3$, Blokhuis proved the conjecture, see \cite[Theorem 6]{Bsurvey}.
		\item For $t=1$, $s=3$, $p\geq 7$, the conjecture follows from \cite{PoSt} by Polverino and Storme, if we put it together with Proposition \ref{PoSt2}.
		\item For $s=2$, the conjecture is true, see \cite{mbs1} by Blokhuis, Storme and Sz\H onyi. 
		\item For $t=2$ there are constructions whose size reaches the order of magnitude in \eqref{bound}, for $s=n$ prime and $p>5$ the exact bound can be reached, see  \cite{dbs2} by De Beule, H\'eger, Sz\H onyi, Van de Voorde; see also \cite{dbs1} by Bacs\'o, H\'eger and Sz\H onyi and \cite{unpub2} by Storme and Polverino.
		\item For $t=3$ and $s=3$, there are constructions for infinitely many $p$  whose size reach the bound \eqref{bound},  \cite{Csajb}. 
	\end{enumerate}
	
	Note that for each set of parameters, the conjecture consists of two parts: proving the lower bound, and presenting constructions whose size attains this bound. We focus on the latter and present constructions for $2$- and $3$-fold blocking sets. In particular, in Section \ref{sec:2} we prove for each $p$ and $n>1$ the existence of two disjoint blocking sets of size $(p^n+p^{n(s-1)/s}+1)$. In Section \ref{sec:3fold} we use the projectivity $(x:y:z) \mapsto (z:x:y)$ to find $3$ disjoint blocking sets of size $(p^{n}+p^{n(s-1)/s}+1)$. 
	The main result of this paper improves the results listed above as \((5)\) and \((6)\).  We prove the following theorem:
	
	\begin{Th}
		For each prime $p$ and integer $n>1$, if $s$ denotes the smallest prime divisor of $n$, then there exist $t$-fold blocking sets of the conjectured size $t(p^n+p^{n(s-1)/s}+1)$ in $\FG$ for $t=2,3$.
	\end{Th}
	
	To prove this, in Section \ref{sec:3fold} we will prove a slightly stronger result:
	\begin{Th}
		\label{3fold_qh}
		Let $q$ be a prime power and $h\geq2$ be an integer. In $\PG(2,q^h)$ there exist $t$-fold blocking sets of size $t(q^h+q^{h-1}+1)$ for $t=2,3$.
	\end{Th}
	
	\section{Preliminaries}
	
	
	In this section, we list the main definitions,  results and constructions for blocking and multiple blocking sets. We present a construction which explains why we assume in Conjecture \ref{conj0} that the size $p^{n/s}$ of the largest subfield of $\F_{p^n}$ is large enough w.r.t. $t$. We will show how the exponent of the smallest nontrivial blocking set is related to the largest subfield of $\F_{p^n}$, cf. Proposition \ref{PoSt2}. Finally, in Proposition \ref{minbs} we prove that the small multiple blocking sets that we are trying to construct are necessarily minimal.

	\subsection{Blocking sets and multiple blocking sets containing lines} 
	
	
	In \cite[Theorem 4.1]{Ball0}, Ball proved that if $\cB$ is a $t$-fold blocking set that contains no line, then it has at least $tq + \sqrt{tq} +1$ points. 
	If $\ell$ is a line contained in a $t$-fold blocking set $\cB$, then $\cB \setminus \ell$ is an affine $(t-1)$-fold blocking set and hence it has size at least $tq-t+1$ (cf. Bruen \cite[Theorem 2.1]{Bruen0}), thus $|\cB|\geq (t+1)q-t+2$. For a fixed $t$, this lower bound is smaller than the conjectured \eqref{bound}  only for small $p^{n/s}$. 
	For $t=2$ this happens only for $q=2^n$, $n$ prime; and for \(t=3\) when $q=2^n$, or $q=3^n$, with \(n\) prime. We show that in these two cases it is in fact possible to construct blocking sets whose size is smaller than \eqref{bound}. These blocking sets, however, contain some lines.
	
	
	\begin{Ex}
		The union of three lines in general position in $\PG(2,q)$ gives a double blocking set of size $3q$. If $q=2^n$, $n$ is a prime, then this number is smaller than \eqref{bound}, which reads as $2(2^n+2^{n-1}+1)$. 
		
		Similarly, if $q$ is even and we take the union of four lines in general position, say \(\{X=0\}\cup\{Y=0\}\cup\{Z=0\}\cup\{X+Y+Z=0\}\), then this point set meets each line of the plane in at least three points, except for the lines \(\{X=Y\},\{Y=Z\}\) and \(\{Z=X\}\), which are blocked twice and are concurrent at the point \((1:1:1)\) (we are using homogeneous coordinates here). And so, \(\{X=0\}\cup\{Y=0\}\cup\{Z=0\}\cup\{X+Y+Z=0\}\cup\{(1:1:1)\}\) is a 3-fold blocking set of size \(4q-1\), which is smaller than \eqref{bound} when $t=3$,  $q=2^n$ and \(n\) is a prime.

		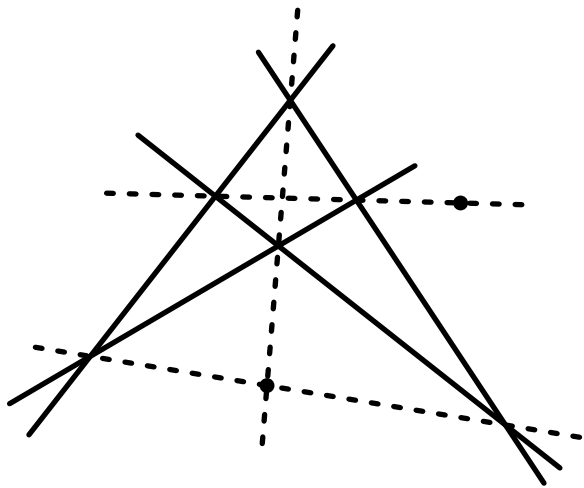
\begin{figure}[h]
			
\begin{center}
	\begin{tikzpicture}[
		line cap=round,
		line join=round,
		>=triangle 45,
		x=1cm,y=1cm,
		scale=0.8,
		transform shape
		]
		
		\tikzset{
			solid line/.style={line width=2pt},
			dashed line/.style={line width=2pt, dash pattern=on 2.5pt off 6pt},
			my point/.style={circle, fill=black, inner sep=0pt, minimum size=7pt}
		}
		
		\draw[solid line] (0.29671444958728077,3.9246867905760423)--(4.993792445111142,-3.217777171689622);
		\draw[solid line] (1.5227568270481144,4.030923276983095)--(-3.4778257477243173,-2.417196358907672);
		\draw[solid line] (-1.6837299644722272,2.5501467223259406)--(5.255768729726721,-2.9659664249919087);
		\draw[solid line] (2.8704922850844965,2.04226304188097)--(-3.795682586333578,-1.9011080088170464);
		
		\draw[dashed line] (0.9426769062316673,4.619782759296475)--(0.3503365552228738,-2.679215776213218);
		\draw[dashed line] (-2.199777455525673,1.5879772401074168)--(4.725433907187155,1.394781660988817);
		\draw[dashed line] (-3.370997554335467,-0.9660193250743894)--(5.658312642352516,-2.4701293623698484);
		
		\node[my point] at (3.6218292173433153,1.4255693923670059) {};
		\node[my point] at (0.4378778707767703,-1.600504976018218) {};
		
	\end{tikzpicture}
\end{center}

			\caption{$3$-fold blocking set of size $4q$}
				\label{Fig1}
			\end{figure}

		Moreover, if we take four lines in general position in \(\PG(2,q)\) with the addition of two suitably chosen points as seen above in Figure \ref{Fig1}, we get a \(3\)-fold blocking set of size \(4q\). When $q=3^n$, $n$ is a prime, then \eqref{conj0} reads as \(4q+3>4q\).
		
	\end{Ex}
	
	This construction can be generalized for arbitrary $t$ to give $t$-fold blocking sets that are smaller than \eqref{bound} when $t$ is large enough w.r.t. $p^{n/s}$. 
	The construction works in an arbitrary projective plane $\Pi_q$ of order $q$. 
	First, take $t+1$ distinct lines $\ell_1, \ell_2, \dots, \ell_{t+1}$ in $\Pi_q$ such that no three of them are incident with the same point.
	These lines will then have $\binom{t+1}{2}$ mutual intersections. Denote the point set of these intersections by $\cM$. The union $\cB:=\ell_1 \cup \ell_2 \cup \dots \cup \ell_{t+1}$ will contain $(t+1)(q+1) - \binom{t+1}{2}$ points.
	This point set meets all lines of the plane in at least $t$ points, except for those incident with multiple of the $\binom{t+1}{2}$ points of $\cM$. Such lines are either the $\ell_i$ themselves, or lines meeting $\cB$ in fewer than $t$ points. We will call the latter lines exceptional.  
	For $i=0,1,2,\ldots,t-1$ let $n_i$ denote the number of lines meeting $\cM$ in exactly $i$ points. 
	By double counting the triples $(P,Q,\ell)$, where $P,Q\in \cM$ are distinct points of the exceptional line $\ell$, we obtain:
	\[|\cM|(|\cM|-2t+1)=\sum_{i=2}^{t-1}i(i-1)n_i.\]
	
	For an exceptional line $\ell$ meeting $\cM$ in $i$ points, we want to add $(i-1)$ extra points of $\ell$ to $\cB$ in order to have $t$ points on that line as well. 
	Hence we must add
	\[\sum_{i=2}^{t-1}(i-1)n_i=
	\sum_{i=3}^{t-1}(i-1)n_i+\frac{1}{2}\left(|\cM|(|\cM|-2t+1)-\sum_{i=3}^{t-1}i(i-1)n_i\right)=\]
	\[\frac12 |\cM|(|\cM|-2t+1)-\sum_{i=3}^{t-1}(i^2+1)n_i\]
	points to $\cB$. In the worst case 
	$n_3=n_4=\ldots=n_{t-1}=0$. 
	This results in $\binom{t+1}{2}\left(\binom{t+1}{2} - 2t + 1\right)/2$ exceptional lines, which are all only incident with $2$ points of $\cM$ and hence $t-1$ points of $\cB$. 
	Each of these lines can then be blocked with the addition of a single point, resulting in a $t$-fold blocking set of size 
	\begin{equation}
		\label{new}
		tq + q + 1 + \left(t^4 - 2t^3 - 5t^2 + 6t\right)/8.
	\end{equation}
	
	If two exceptional lines meet each other outside of $\cB$, then they can be blocked by a single point instead of using two points. This means that a more careful analysis could slightly improve the construction above. 
	
	Note that extending with points the union of some lines in general position is a standard technique to construct small multiple blocking sets. For some recent constructions of this type see for example \cite{Daskalov}.
	

	
	

	
	\subsection{Linear blocking sets}
	\label{Sec:2.2}
	
	\begin{Def} \label{def:small}
		A blocking set $\cB$ in \(\PG(2,q)\) is called small if $|\cB|< 3(q+1)/2$. 
		A \(k\)-secant of $\cB$ is a line that meets $\cB$ in exactly $k$ points.
		Blocking sets of size $q+N\leq2q$ in $\PG(2,q)$ with an $N$-secant are called \textit{R\'edei type blocking sets}. Lines that meet such a blocking set in $N$ points are called \textit{R\'edei lines}. 
		A \(t\)-fold blocking set of \(\PG(2,q)\) is called small when it has less than \(tq +(q +3)/2\) points.
	\end{Def}
	
	In our constructions, we will use disjoint copies of small linear blocking sets of R\'edei type. Linear blocking sets are of interest to us since, over prime fields, it has been proved, while over fields of composite order it is conjectured that all small blocking sets are linear. 
	
	The theory of linear blocking sets is based on the fact that the lattice of $\F_{q^s}$-subspaces of $V=\F_{q^s}^3$ is isomorphic to $\Pi=\PG(2,q^s)$. Lines of $\Pi$ are set of points defined by nonzero vectors of two-dimensional $\F_{q^s}$-subspaces of $V$. 
	Let $U$ be any $\F_q$-subspace of $V$ of dimension $(s+1)$, denote by $U^*$ the set of nonzero vectors of $U$, and let $T$ be any two-dimensional $\F_{q^s}$-space. By Grassmann's identity, $U\cap T$ is nontrivial, and hence
	\[\cL_U:=\{ \la\vec v \ra_{\F_{q^s}} : \vec  v\in U^* \}\]
	is a blocking set, which is called an $\F_q$\textit{-linear blocking set} (of rank $(s+1)$). 
	This argument is due to Lunardon \cite{Lunardon}.
	
	\begin{Def}
		The weight of the point $P=\la\vec v \ra_{\F_{q^s}}$ w.r.t. $\cL_U$ is defined to be \(w(P):=\dim_{\Fq}(U\cap \la\vec v \ra_{\F_{q^s}})\). Similarly, if $\ell$ is a line whose points are defined by the nonzero vectors of the $2$-dimensional $\F_{q^s}$-subspace $S$, then the weight of $\ell$ w.r.t. $\cL_U$ is $w(\ell):=\dim_{\F_q}(U \cap S)$. 
	\end{Def}
	
	The following example of a linear blocking set will be crucial for us.
	
	\begin{Ex}[The trace construction]
		\label{traceex}
		If $s$ is the smallest prime divisor of $n>1$, then $\F_{p^{n/s}}$ is the largest subfield of $\F_{p^n}$. Put $q:=p^{n/s}$. 
		In $\PG(2,p^n)=\PG(2,q^s)$ consider the  mapping $\Tr_{q^s/q}: \F_{q^s} \rightarrow \F_{q}$ defined by $\Tr_{q^s/q}(x):=x+x^q+\ldots+x^{q^{s-1}}$. Note that this is an $\F_{q}$-linear functional. 
		
		We may take the following $(s+1)$-dimensional $\F_q$-subspace of $V=\F_{q^s}^3$: 
		\begin{equation}
			\label{B}
			U:=\{(x,\Tr_{q^s/q}(x),y) : x \in \F_{q^s}, y \in \F_q\}.
		\end{equation}
		It is well-known that $\cL_U=\{(x:\Tr_{q^s/q}(x):y) : x \in \F_{q^s}, y \in \F_q, (x,y)\neq (0,0)\}$ is a blocking set of size $q^s+q^{s-1}+1$ in $\PG(2,q^s)$ (cf. \cite[pg.\ 138]{Bsurvey}), thus when \(s\) is the smallest prime divisor of \(n\), its size reaches the lower bound in Conjecture \ref{conj0}. 
		
		The point $(1:0:0)$ is defined by the vectors in $\{(x,0,0) : \Tr_{q^s/q}(x)=0\}$ and hence it has 
		weight $(s-1)$. 
		Also, $\cL_U$ is a blocking set of R\'edei type, $\ell_{\infty}$  (the line of equation $Z=0$) is one of its R\'edei lines. This means that $\cL_U$ consists of the graph of the trace function in the affine plane (points of form $(x:\Tr_{q^s/q}(x):1)$), and the directions determined by this graph (points of form $(x:\Tr_{q^s/q}(x):0)$).
	\end{Ex}
	
	From now on, we will refer to this construction as the trace construction.
	Note that the observations made for Example \ref{traceex} are true if we replace \(\Tr_{q^s/q}\) with any $\F_q$-linear function whose kernel has codimension $1$. These are exactly the maps $x \mapsto \alpha \Tr(\beta x)$, where $\alpha,\beta \in \F_{q^s}^*$. In our constructions, we will use $\F_{q^s}\rightarrow\Fq$ linear functionals. The following proposition is crucial to our cause.
	
	\begin{Prop}
		\label{prop:trace}
		Let $W$ be any $(s+1)$-dimensional $\F_q$-subspace of $V=\F_{q^s}^3$ such that $\cL_W$ is a nontrivial (i.e., it is not a line) blocking set of $\PG(2,q^s)$ with a point $P$ of weight $(s-1)$. Then $\cL_W$ is projectively equivalent to $\cL_U$ defined in Example \ref{traceex}.
	\end{Prop}
	\begin{proof}
		First, note that lines have weight at most $s$ w.r.t. $\cL_W$. Indeed, if there were a line $\ell$ of weight at least $s+1$, then, by Grassmann's identity, each point of $\ell$ is a point of $\cL_W$, and hence $\cL_W$ is trivial. 
		If $Q\neq P$ was another point of $\cL_W$
		of weight larger than $1$, then the line $\langle Q,P\rangle$ would have weight $s+1$, a contradiction. 
		Let $\ell$ denote a line through $P$ and put $S$ for the $2$-dimensional $\F_{q^s}$-subspace of $V$ whose nonzero vectors define the points of $\ell$. Then either $\cL_W \cap \ell=\{P\}$, or  $\dim_{\F_q}(S \cap W)=s$. In the latter case, put $w_1$ to denote the number of points of weight $1$ in $\ell \cap \cL_W$. Then counting nonzero vectors of $S\cap W$ gives 
		$q^{s}-1=|(S\cap W)^*|=w_1(q-1)+(q^{s-1}-1)$ and hence $w_1=q^{s-1}$, thus $|\ell \cap \cL_W|=w_1+1=q^{s-1}+1$. 
		
		After a suitable projectivity we may assume $P=(1:0:0)$, $O=(0:0:1)\in \cL_W$, $(\infty)=(0:1:0)\notin \cL_W$ and that $\ell_{\infty}$ is a line meeting $\cL_W$
		in $q^{s-1}+1$ points.
		Then $W$ is generated by $(0,0,1)$ and by $\{(x,f(x),0) : x\in \F_{q^s}\}$ for some $\F_{q}$-linear function $f \colon \F_{q^s} \rightarrow \F_{q^s}$, i.e.  
		\[\cL_W=\{ (x : f(x) : y) : x\in \F_{q^s}, y \in \F_q, (x,y)\neq (0,0)\}.\]
		The line $\la O, P\ra$ is a $(q^{s-1}+1)$-secant of $\cL_W$ and hence the kernel of $f$ has dimension $(s-1)$ over $\F_q$. It follows that $f(x)=\alpha \Tr(\beta x)$ for some $\alpha,\beta \in \F_{q^s}^*$. The projectivity $(x:y:z) \mapsto (\beta x: y/\alpha:z)$ maps $\cL_W$ to $\cL_U$. 
	\end{proof}
	
	So, any $(s+1)$-dimensional $\F_q$-subspace of $V=\F_{q^s}^3$ that satisfies the conditions of the proposition defines a linear blocking set of desired size. Our goal is to find such sets disjoint from each other. In Proposition \ref{minbs}, we show that the disjoint union of three copies of the trace construction is necessarily a minimal $3$-fold blocking set.
	
	\color{blue}
	\color{black}
	\begin{Def}\label{minimal}
		A \(t\)-fold blocking set \(\cB\) is called minimal if it does not contain a smaller \(t\)-fold blocking set, or equivalently, if each of its points is contained in a \(t\)-secant of $\cB$.  
		
		If $\cB$ is a small minimal nontrivial blocking set, then Sz\H onyi proved that there is an integer $1 \le e < n$ such that each line meets $\cB$ in $1 \pmod{p^e}$ points, and there is a line meeting $\cB$ not in $1 \pmod{p^{\,e+1}}$ points, see \cite{1modp}. 
		The integer $e$ is called the \emph{exponent} of $\cB$, and Sziklai proved in \cite{linconj} that $e \mid n$.
	\end{Def}
	
	\begin{Prop}
		\label{PoSt2}
		Let $\cB$ be the smallest minimal nontrivial blocking set in $\PG(2,p^n)$, $n>1$, and let $s$ denote 
		the smallest prime divisor of $n$. Then the exponent $e$ of $\cB$ is $n/s$.
	\end{Prop}
	\begin{proof}
		Since $e \mid n$, the statement is trivial when $n$ is prime, so assume that $n$ is composite.  
		Sz\H onyi proved in \cite{1modp} the inequality
		\[
		p^n + 1 + \frac{p^n}{p^e + 2} \;\leq\; |\cB|.
		\]
		Since $e<n$ and $e \mid n$, we must have $e \leq n/s$.  
		On the other hand, Example~\ref{traceex} shows the existence of minimal nontrivial blocking sets of size
		\[
		p^n + p^{\,n(s-1)/s} + 1.
		\]
		Thus, it remains to show that if $e < n/s$, then
		\[
		p^n + p^{\,n(s-1)/s} + 1 
		\;<\;
		p^n + 1 + \frac{p^n}{p^e + 2}.
		\]
		The right-hand side is the smallest when $e$ is the largest, i.e.\ when $e=n/s - 1$.  
		Substituting $e = n/s - 1$ and simplifying gives the inequality
		\[
		p^{\,n-1}\!\left(2p^{\,1-n/s} - p + 1\right) < 0,
		\]
		which is equivalent to $p^{\,1-n/s} < \tfrac{p-1}{2}$.  
		The left-hand side is maximal when $n/s$ is minimal.  
		Since $n$ is composite, $1 - n/s \leq -1$, so $p^{\,1-n/s} \leq p^{-1}$, and the inequality 
		$p^{-1} < (p-1)/2$ holds for all $p>2$, and also for $p=2$ whenever $n/2 > s$.
		
		If $s = n/2$, then $n = 2s$ and $n/2$ is the smallest prime divisor of $n$.  
		This forces $n/2 = 2$, so $n = 4$.  
		In this case, $\cB$ has the size of a Baer subplane in $\PG(2,16)$ and therefore is itself a 
		Baer subplane \cite{Bruen}.  
		Its exponent is $2 = n/s$, which completes the proof.
	\end{proof}

	When $t=1$, $s=3$, $p \ge 7$, then Conjecture \ref{conj0} follows from \cite{PoSt} together with Proposition~\ref{PoSt2}. 
	Indeed, in the case of $s=3$, the exponent of the smallest minimal nontrivial blocking sets is $e=n/3$, see 
	Proposition~\ref{PoSt2}. 
	Polverino and Storme proved in \cite{PoSt} that every small minimal blocking set 
	with exponent $n/3$ has size at least $p^n + p^{2n/3} + 1$, and that this lower bound is sharp. 
	Thus, the conjectured bound~\eqref{bound} holds for $t=1$, $s=3$, $p \ge 7$.
	
	\medskip
	
	The following lemma is useful when dealing with linear functionals:
	
	\begin{Lemma}\label{f_counting_lemma}
		Let $f \colon \Fqh\rightarrow\Fq$ be a nonzero, $\Fq$-linear functional. Then for any $k\in\Fq$, \[
		\lvert\{x\in\Fqh : f(x)=k\}\rvert=q^{h-1}.
		\]\qed 
	\end{Lemma}
	
	In the next result, we summarize the main properties of the trace construction. These are well-known results and they can be proved easily by using Lemma \ref{f_counting_lemma} and some linear algebra. 
	We will need this lemma to prove the minimality of our future constructions.
	
	
	
	\begin{Lemma}
		\label{tracelemma}
		Let $\cL$ denote the blocking set $\{(x:\Tr_{q^h/q}(x):y) : x\in \F_{q^h},y \in \F_q, (x,y)\neq (0,0)\}$ of $\PG(2,q^h)$, $h\geq 2$.
		\begin{enumerate}[\rm(1)]
			\item Each line meets $\cL$ in $1$, $(q+1)$, or in $(q^{h-1}+1)$ points.
			\item $(1:0:0)$ is the the only point of $\cL$ of weight $(h-1)$ and it is incident with $(q^h-q)$ tangent lines to $\cL$ and with $(q+1)$ lines meeting $\cL$ in $(q^{h-1}+1)$ points. These are the only $(q^{h-1}+1)$-secants of $\cL$.
			\item The other points of $\cL$ are of weight $1$ and they are incident with $q^h-q^{h-1}$ tangent lines to $\cL$. \qed
		\end{enumerate}  
	\end{Lemma}
	
	\begin{Lemma}\label{tangent_outside_lemma}
		Let $\cL$ be a blocking set in $\FGq$, $h\geq 2$,  equivalent to the trace construction, $P$ a point not in $\cL$. Then $P$ is incident with at least $q^h-q^{h-2}+1>0$ tangent lines to $\cL$.
	\end{Lemma}
	\begin{proof}
		Let $k$ be the number of tangents to $\cL$ incident with $P$. Counting the number of points of $\cL$ on the $q^h+1$ lines incident with $P$ gives the following:
		\begin{enumerate}[\rm(1)]
			\item If $P$ is not incident with $(q^{h-1}+1)$-secants of $\cL$, then:
			\[
			1\cdot k+(q+1)(q^h+1-k)=q^h+q^{h-1}+1.
			\]
			Solving this linear equation for $k$, we get $k=q^h-q^{h-2}+1$.
			\item If $P$ is incident with a unique $(q^{h-1}+1)$-secant of $\cL$, then:
			\[
			(q^{h-1}+1)\cdot1+1\cdot k+(q+1)(q^h-k).
			\]
			Solving for $k$, we get $k=q^h$.
			\item If $P$ is incident with at least two $(q^{h-1}+1)$-secants of $\cL$, then by Lemma \ref{tracelemma} $(2)$, $P\in \cL$, a contradiction.
		\end{enumerate}
		The result follows, since $q^h\geq q^h-q^{h-2}+1$ for $h\geq2$.
	\end{proof}
	
	\begin{Prop}\label{2fold_minimal}
		Let $\cL_1$ and $\cL_2$ be two disjoint blocking sets in $\FGq$, $h\geq 2$, both projectively equivalent to the trace construction. Then each point of $\cL_1\cup \cL_2$ is incident with at least $(q^h-q^{h-1}-q^{h-2})$ $2$-secants. In particular, $\cL_1\cup \cL_2$ is a minimal double blocking set.
	\end{Prop}
	\begin{proof}
		Since $\cL_1\cap \cL_2=\emptyset$, it is clear that $\cL_1\cup \cL_2$ is a 2-fold blocking set. 
		
		Take a point $P$ from $\cL_1$. By Lemma \ref{tracelemma}, $P$ is incident with at least $q^h-q^{h-1}$ tangent lines to $\cL_1$, and, by Lemma \ref{tangent_outside_lemma}, $P$ is incident with at least $q^h-q^{h-2}+1$ tangent lines to $\cL_2$. Since $P$ is incident with $(q^h+1)$ lines, according to the inclusion-exclusion principle, there are at least \[
		(q^h-q^{h-1}) + (q^h-q^{h-2}+1)-(q^h+1)=q^h-q^{h-1}-q^{h-2}
		\]
		lines incident with $P$ that are tangent to both $\cL_1$ and $\cL_2$. Interchanging the role of $\cL_1$ and $\cL_2$ yields the same result for points of $\cL_2$. Since $q^h-q^{h-1}-q^{h-2}$ is positive, each point of $\cL_1\cup \cL_2$ is incident with a $2$-secant, and hence it is a minimal double blocking set.
	\end{proof}
	
	\begin{Prop}
		\label{minbs}
		Let $q>2$ be a prime power, $h\geq 2$, and $\cL_1,\cL_2,\cL_3$  be pairwise disjoint blocking sets in $\FGq$, all projectively equivalent to the trace construction. Then $\cL_1\cup \cL_2\cup \cL_3$ is a minimal $3$-fold blocking set.
	\end{Prop}
	\begin{proof}
		Again we see that $\cL_1\cup \cL_2\cup \cL_3$ is clearly a 3-fold blocking set.\\
		Take a point $P$ from $\cL_1$. According to  Proposition \ref{2fold_minimal}, there are at least $q^h-q^{h-1}-q^{h-2}$ $2$-secants of $\cL_1\cup L_2$ incident with $P$. By Lemma \ref{tangent_outside_lemma}, $P$ is on at least $q^h-q^{h-2}+1$ tangent lines to $\cL_3$. This means that there are at least \[(q^h-q^{h-1}-q^{h-2})+(q^h-q^{h-2}+1)-(q^h+1)=q^h-q^{h-1}-2q^{h-2}\] lines incident with $P$ that are $2$-secants of $\cL_1\cup \cL_2$ and tangents to $\cL_3$, i.e. they are $3$-secants of $\cL_1\cup \cL_2\cup \cL_3$. If $q>2$, then his number is positive. Again, using the same argument for points of $\cL_2$ and $\cL_3$, we see that $\cL_1\cup \cL_2\cup \cL_3$ is a minimal $3$-fold blocking set.
	\end{proof}
	
	\section{Double blocking sets}\label{sec:2}
	
	\subsection{The upper chromatic number of finite projective planes}
	\label{subsec:chrom}
	
	Before moving on to the construction of a $2$-fold blocking set (also known as a double blocking set), we provide some motivation as to why the minimum size of double blocking sets matters. We take a look at the proper colorings of finite projective planes and their connection to $2$-fold blocking sets.  The results presented in this section are based on the work of Bacs\'o, H\'eger and Sz\H onyi. For further information on the topic, we recommend reading \cite{dbs1}.\\
	
	A \textit{hypergraph} \(\mathcal{H}\) consists of an underlying set \(X\) of vertices and a system of sets on \(X\) that define the hyperedges of the graph. A coloring assigns each of the vertices a color. A \textit{proper coloring} is a coloring such that each hyperedge has at least two vertices that have the same color.
	
	\begin{Def}
		The upper chromatic number of \(\mathcal{H}\) is the maximum number of colors that can be used in a proper coloring. Notation: \(\overline{\chi}(\mathcal{H})\).
	\end{Def}
	
	A finite projective plane $\Pi$ is a hypergraph, the underlying set of vertices being the points of the plane, and the hyperedges being defined by the lines. If we take a coloring of the plane with colors \(c_1,\dots,c_k\), then the sets \(C_i\), the set of points colored \(c_i\), partition the plane's point set. Moreover, if the coloring is proper, then the union \(\bigcup_{|C_j|\geq2} C_{j}\) is a double blocking set. Conversely, a double blocking set defines a proper coloring: assign the same color to all points of the double blocking set, and distinct colors to all points outside it.  
	When the double blocking set is of minimum size (\(\tau_2\)), then these proper colorings are called \textit{trivial colorings}. Such colorings use \(v-\tau_2+1\) colors, where $v$ is the number of points of $\Pi$. Thus,
	\[\overline{\chi}(\Pi)\geq v-\tau_2+1.\]
	
	Furthermore, it was shown that if \(q>256\) is a square,
	or \(p>29\) and \(h\geq3\) is odd, then \(\overline{\chi}(\PG(2,p^h))= v-\tau_2+1\), see \cite[Theorem 1.12]{dbs1}. So, the value of \(\tau_2\) is of particular interest.
	
	\subsection{Construction of two disjoint blocking sets}
	
	As stated before, we aim to construct a $2$-fold blocking set of size $2(q^h+q^{h-1}+1)$ in $\FGq$ by finding two disjoint blocking sets, each of size $q^h+q^{h-1}+1$. Then, their union will be a double blocking set of the desired size.\\\\
	The following lemma will be useful:\begin{Lemma}\label{uvw}
		Let $\vec u, \vec v, \vec w$ be three linearly independent vectors in $\Fqh^3$, and let $f \colon \Fqh\rightarrow \Fq$ be a nonzero linear functional. Define the subspace \[U:=\{x\vec u+f(x)\vec v+y\vec w :  x\in\Fqh,y\in\Fq\}.\] Then $\cL_U$ is a blocking set of size $q^h+q^{h-1}+1$.
		
	\end{Lemma}
	\begin{proof}
		Since \(\dim_{\Fq} U=h+1\), by Proposition \ref{prop:trace}, it is sufficient to prove that \(\cL_U\) is a nontrivial blocking set with a point of weight \(h-1\). A line must have weight $(h+1)$ to be contained in $\cL_U$, so $\cL_U$ is trivial if and only if it is a line.  
		There exists an \(a\in\Fqh\) such that \(f(a)\neq0\). Then the nontriviality of \(\cL_U\) follows from the fact that the points \(\la\vec u \ra_{\F_{q^h}}\), \(\la a\vec u+f(a)\vec v \ra_{\F_{q^h}}\) and \(\la\vec w \ra_{\F_{q^h}}\) of $\cL_U$ are not collinear, since \(\vec u,\vec v\) and \(\vec w\) were chosen to be linearly independent over \(\Fqh\).
		
		The point \(\la\vec u \ra_{\F_{q^h}}\) in \(\cL_U\) has weight \(h-1\), because \[U\cap \la\vec u \ra_{\F_{q^h}}=\{x\vec u+f(x)\vec v+y\vec w : x\in\Fqh,f(x)=y=0\}\] is an $(h-1)$-dimensonal \(\Fq\)-subspace.
		
		And so, \(\cL_U\) is projective equivalent to the blocking set presented in Example \ref{traceex}, which has size \(q^h+q^{h-1}+1\).
	\end{proof}
	
	\begin{Rem} \label{redei}
		Note that because of the $\F_q$-linearity of $f$, if $y\not=0$, then one can assume $y=1$. If we have $\vec u=(1,0,0),\vec v=(0,1,0),\vec w=(0,0,1)$, then the points with $y=0$ and $y=1$ correspond to the ideal and affine points of the blocking set, respectively.
	\end{Rem}
	
	We will search for two disjoint blocking sets of the following form:\[
	\cL=\{(x:f(x):y) : x\in \Fqh,y\in\Fq, (x,y)\neq (0,0)\},\]
	\[\cL'=\{(y':x':g(x')+y'\alpha) : x'\in\Fqh,y'\in\Fq, (x',y')\neq (0,0)\},
	\]
	where $f,g \colon \Fqh\rightarrow\Fq$ are nonzero linear functionals, and $\alpha\in\Fqh\setminus\Fq$ is a constant. Note that $\cL'$ is also a blocking set of size $q^h+q^{h-1}+1$ because of Lemma \ref{uvw} ($\vec u=(0,1,0),\vec v=(0,0,1),\vec w=(1,0,\alpha)$).
	
	\begin{Th}\label{2fold_construction}
		The blocking sets $\cL$ and $\cL'$ are disjoint if the following conditions hold:\begin{enumerate}[\rm(1)]
			\item $g(1)=g(\alpha)=1$,
			\item $f\left(\frac1\alpha\right)\not=0$,
			\item $f\left(\frac{k-1}{k(k-1)\alpha-k^2}\right)\not=1$ for all $k\in\Fq$, $k\not=0,1$.
		\end{enumerate}
	\end{Th}
	\begin{proof}
		Let $P\in \cL$, $P'\in \cL'$, meaning \[
		P=(x:f(x):y), \quad P'=(y':x':g(x')+y'\alpha),
		\] for some $x,x'\in\Fqh$, $y,y'\in\Fq$. Assume that the conditions hold. We want to show $P\not=P'$. We will do this by checking case by case based on whether $y$ and $y'$ are $0$ or $1$.
		We begin by indirectly assuming $P=P'$.
		\begin{enumerate}
			\item $y=0$, meaning \[(x:f(x):0)=(y':x':g(x')+y'\alpha).\]
			Then $g(x')+y'\alpha=0$, but because of $\alpha\not\in\Fq$, that means $g(x')=y'=0$, so $P'=(0:x':0)=(0:1:0)$. This means $x=0$ and $f(x)\not=0$, a contradiction.
			\item $y'=0$, meaning \[(x:f(x):y)=(0:x':g(x')).\]
			Therefore $x=0$, so $P=(0:f(0):y)=(0:0:1)$, but then $x'=0$ and $g(x')\not=0$, a contradiction.
			\item $y=y'=1$, meaning \[(x:f(x):1)=(1:x':g(x')+\alpha).\]
			We introduce $k=f(x)$ and $l=g(x')$ (note that both $k$ and $l$ are in $\Fq$). Now $P=P'$ exactly if \[
			x=\frac1{l+\alpha}\text{, and } x'=k(l+\alpha)
			.\] 
			Applying $f$ and $g$ to these equations, we get
			\begin{equation}
				\label{first}
				k=f\left(\frac1{l+\alpha}\right)
			\end{equation}        
			and 
			\[l=g(kl+k\alpha)=klg(1)+kg(\alpha)=kl+k.\]        
			The second equation can be rearranged as 
			\begin{equation}
				\label{second}
				(k-1)l=-k.
			\end{equation}
			Now we show that, for every $k\in\Fq$, we obtain a contradiction.
			\begin{enumerate}
				\item $k=0$: From \eqref{second} we get $l=0$. Substituting this into \eqref{first} yields $0=f\left(\frac1\alpha\right)$, which contradicts condition $(2)$ of the statement.
				\item $k=1$: Substituting into \eqref{second}, we obtain $k=0$, a contradiction.
				\item $k\not=0,1$: Dividing  \eqref{first} by $k$, we obtain
				\[1=f\left(\frac{1}{k(l+\alpha)}\right)=f\left(\frac{k-1}{k(k-1)\alpha+kl(k-1)}\right)=f\left(\frac{k-1}{k(k-1)\alpha-k^2}\right),\] which contradicts  condition $(3)$ of the statement.
			\end{enumerate}
		\end{enumerate}
	\end{proof}
	
	\begin{Prop}\label{2fold_existence}
		For any nonzero $f:\Fqh\rightarrow\Fq$, there exist $\alpha$ and $g$ such that the conditions in Theorem \ref{2fold_construction} hold.
	\end{Prop}
	\begin{proof}
		First, we show that there exists an $\alpha\in\Fqh\setminus\Fq$ such that the second and third conditions in Theorem \ref{2fold_construction} hold. Then we choose a suitable $g$.
		
		We denote by $B_0$ the set of $\alpha\in\Fqh$ where the second condition does not hold, and by $B_k$ the set of $\alpha\in\Fqh$ where the third condition fails for a fixed $k\in\Fq$, $k\neq 0,1$. 
		We count the size of these sets using Lemma \ref{f_counting_lemma}:
		\[\lvert B_0\rvert=\left\lvert\left\{\alpha\in\Fqh :  f\left(\frac{1}{\alpha}\right)=0\right\}\right\rvert=\left\lvert\left\{\beta : \beta\in\Fqh,  f\left(\beta\right)=0\right\}\setminus\{0\}\right\rvert
		=
		q^{h-1}-1.\]
		\[\lvert B_k\rvert=\left\lvert\left\{\alpha\in\Fqh\ |\ f\left(\frac{k-1}{k(k-1)\alpha-k^2}\right)=1\right\}\right\rvert=\left\lvert\left\{\beta : \beta\in\Fqh,  f\left(\beta\right)=1\right\}\right\rvert
		=q^{h-1},\] 
		for any $k\in\Fq$, $k\not=0,1$. This is because the $\Fqh\setminus\{k/(k-1)\}\rightarrow \Fqh\setminus\{0\}$ function $\alpha\mapsto\frac{k-1}{k(k-1)\alpha-k^2}$ is a bijection for each $k\in \F_q \setminus \{0,1\}$. 
		
		Therefore, \begin{multline*}
			\left\lvert\left\{\alpha\in\Fqh\setminus\Fq : f\left(\frac{1}{\alpha}\right)\not=0,\, f\left(\frac{k-1}{k(k-1)\alpha-k^2}\right)\not=1\ \ \forall k\in\Fq\setminus\{0,1\}\right\}\right\rvert=\\=\left\lvert\Fqh\setminus\left(\Fq\cup\bigcup_{k\not=1}B_k\right)\right\rvert\geq q^h-q-(q-1)q^{h-1}+1=q^{h-1}-q+1>0.
		\end{multline*}
		
		We have shown that there always exists an $\alpha$ such that the second and third conditions of Theorem \ref{2fold_construction} are met. It remains to find an $\F_q$-linear $g \colon \Fqh\rightarrow\Fq$ such that $g(1)=g(\alpha)=1$. A $g$ like this exists, because to obtain an $\F_{q^h} \rightarrow \F_q$ linear functional, the values  can be  freely chosen on an $\F_q$-linearly independent set, and $\{1,\alpha\}$ is such a set because $\alpha\not\in\Fq$. If $g(1)=g(\alpha)=1$, then any such $g$ is obviously nonzero.
	\end{proof}
	
	\begin{Cor}\label{2fold_main}
		For all prime powers $q$, there exists a $2$-fold blocking set of size $2(q^h+q^{h-1}+1)$ in $\FGq$, $h\geq 2$.
	\end{Cor}
	
	\section{\texorpdfstring{$3$-fold}{3-fold} blocking sets}
	\label{sec:3fold}

	\subsection{Preliminaries to the main construction}
	
	Similarly to what we did in the case of $2$-fold blocking sets, we now construct three disjoint blocking sets of size $q^h+q^{h-1}+1$ in $\FGq$, where $h\geq2$. We will not concern ourselves with the case of $h=2$, since in that case one can just pick three disjoint Baer subplanes from a partition of the point set of $\PG(2,q^2)$ into Baer subplanes, see e.g. \cite{BSP}.
	
	\medskip
	
	Let us consider a blocking set of the form
	\[
	\cL:=\{(x:f(x)+y\alpha:y\beta) : x\in\Fqh,y\in\Fq, (x,y)\neq (0,0)\},
	\]
	where $f:\Fqh\rightarrow\Fq$ is a nonzero  $\Fq$-linear functional, $\alpha\in\Fqh\setminus\Fq$, $\beta\in\Fqh\setminus\{0\}$.
	The point set $\cL$ is a blocking set of size $q^h+q^{h-1}+1$ by Lemma \ref{uvw} with $\vec u = (1,0,0)$, $\vec v=(0,1,0)$, and $\vec w=(0,\alpha,\beta)$.

	Next, we take the collineation \[\varphi:\FGq\rightarrow\FGq,\,(x:y:z)\mapsto(z:x:y).\]
	Note that $\varphi$ has order $3$. Since $\varphi$ is a collineation, $\varphi(\cL)$ and $\varphi^2(\cL)$ are also blocking sets of size $q^h+q^{h-1}+1$. Our aim is to choose $\cL$ such that $\cL$, $\varphi(\cL)$ and $\varphi^2(\cL)$ are pairwise disjoint. Notice that if $\cL\cap\varphi(\cL)=\emptyset$, then applying $\varphi$ yields $\varphi(\cL)\cap\varphi^2(\cL)=\varphi(\emptyset)=\emptyset$, and applying $\varphi$ again, we get $\varphi^2(\cL)\cap\varphi^3(\cL)=\varphi^2(\cL)\cap \cL=\emptyset$. This means that it suffices to check that $\cL\cap\varphi(\cL)=\emptyset$  in order to obtain a $3$-fold blocking set.
	
	\medskip
	
	Similarly to Theorem \ref{2fold_construction}, we set conditions on $f$, $\alpha$ and $\beta$ such that if they hold, then $\cL$ and $\varphi(\cL)$ are disjoint:
	
	\begin{Th}\label{3fold_construction}
		The blocking sets $\cL$ and $\varphi(\cL)$ (defined as above) are disjoint if the following conditions hold:\begin{enumerate}[\rm(1)]
			\item $f\left(\frac{\alpha}{\beta}\right)\not=1$,
			\item $k\not=f\left(\frac{\alpha^2}{\beta}\right)+\left(f\left(\frac{\beta^2}{\alpha+k}\right)+k\right)f(\frac\alpha\beta) + kf\left(\frac{\beta^2}{\alpha+k}\right)f\left(\frac1\beta\right)$ for all $k\in \mathbb F_q$.
		\end{enumerate}
	\end{Th}
	\begin{proof}
		The proof is analogous to that of Theorem \ref{2fold_construction}. We take $P\in \cL$ and $P'\in \varphi(\cL)$, meaning 
		\[P=(x:f(x)+y\alpha:y\beta), \quad P'=(y'\beta:x':f(x')+y'\alpha),\] for some $x,x'\in\Fqh$, $y,y'\in\Fq$. We want to show $P\not=P'$ by indirectly assuming $P=P'$. We will find contradictions case by case based on whether $y$ and $y'$ are $0$ or $1$.
		\begin{enumerate}
			\item $y'=0$, meaning \[
			(x:f(x)+y\alpha:y\beta)=(0:x':f(x')).
			\]
			The first coordinate has to be zero, meaning $x=0$. But then $y\not=0$ (because $(0:0:0)$ is not a point), so we can take $y=1$, and we get $(0:x':f(x'))=(0:\alpha:\beta)$, or  equivalently, $\left(0: \frac{x'}{f(x')}:1\right)=(0:\frac\alpha\beta:1)$.
			It follows that 
			$f\left(\frac{x'}{f(x')}\right)=f\left(\frac{\alpha}{\beta}\right)$. Now, the fact that $f$ is an $\Fq$-linear functional implies $f\left(\frac\alpha\beta\right)=1$, which contradicts the first condition.
			\item $y=0$ and $y'=1$, meaning\[
			(x:f(x):0)=(\beta:x':f(x')+\alpha).
			\]
			This is impossible, since $f(x')\in\Fq$, $\alpha\not\in\Fq$, and hence $f(x')+\alpha\neq 0$.
			\item $y=y'=1$, meaning \[
			(x:f(x)+\alpha:\beta)=(\beta:x':f(x')+\alpha).
			\]
			We introduce $k=f(x')$ and $l=f(x)$ (note that both $k$ and $l$ are in $\Fq$).  
			After cross-multiplying the third coordinates by the first and second coordinates, respectively, and solving the equations for $x$ and $x'$, we can see that $P=P'$ implies
			\[
			x=\frac{\beta^2}{k+\alpha}\text{, and } x'=\frac{(k+\alpha)(l+\alpha)}{\beta}=kl\frac1\beta+(l+k)\frac\alpha\beta+\frac{\alpha^2}\beta.\]  
			Applying $f$ to both equations, we obtain:
			\[
			l=f\left(\frac{\beta^2}{k+\alpha}\right)\text{, and }k=klf\left(\frac{1}{\beta}\right)+(l+k)f\left(\frac{\alpha}{\beta}\right)+f\left(\frac{\alpha^2}{\beta}\right).
			\] Substituting $l=f\left(\frac{\beta^2}{\alpha+k}\right)$ into the second equation yields:\[
			k=f\left(\frac{\alpha^2}{\beta}\right)+\left(f\left(\frac{\beta^2}{\alpha+k}\right)+k\right)f\left(\frac\alpha\beta\right) + kf\left(\frac{\beta^2}{\alpha+k}\right)f\left(\frac1\beta\right).
			\]
			Since $k\in\Fq$, this contradicts the second condition of the statement.
		\end{enumerate} 
	\end{proof}
	
	\begin{Rem}
		If such a $3$-fold blocking set exists, then of course there exists a $2$-fold blocking set as well, by picking two of the disjoint blocking sets.
	\end{Rem}
	
	The conditions of this theorem are too complicated for a simple counting argument (as in Proposition \ref{2fold_existence}) to work. In the next two subsections, we show the existence of $\alpha,\beta$ and $f$ for which the conditions of Theorem \ref{3fold_construction} hold with $h\geq 3$.
	
	\subsection{The case of \texorpdfstring{$h\geq4$}{h>=4}.} We will use Theorem \ref{3fold_construction} to show that a $3$-fold blocking set of size $3(q^h+q^{h-1}+1)$ exists in $\FGq$ for $h\geq 4$. Therefore, we want to find an $\alpha\in\Fqh\setminus\Fq$, $\beta\in\Fqh\setminus\{0\}$ and an $f \colon \Fqh\rightarrow\Fq$ nonzero linear functional satisfying the conditions of the theorem.
	
	\medskip
	
	Take any $\alpha\in\Fqh$ such that $\Fqh=\Fq(\alpha)$, and let 
	\[m(x)=x^h+\lambda_{h-1}x^{h-1}+\dots+\lambda_1x+\lambda_0\] be its minimal polynomial over $\Fq$. Put $\beta:=1$ and define $f$ on the basis $1,\alpha,\dots,\alpha^{h-1}$ by:\begin{equation}\label{hgeq4_f_def}
		f\left(\alpha^j\right):=
		\begin{cases}
			1\quad \text{if} \quad j=0\\
			s\quad \text{if} \quad j=2\\
			t\quad \text{if} \quad j=h-1\\
			0\quad \text{otherwise}
		\end{cases},
	\end{equation} where $s,t\in\Fq$ will be chosen later. In fact, we aim to show that there exist at least $q-1$ choices for $(s,t)$ such that the conditions of Theorem \ref{3fold_construction} hold. 
	For this, we need to calculate $\frac1{\alpha+k}$ as the linear combination of the $\F_q$-basis $1,\alpha,\dots,\alpha^{h-1}$ of $\F_{q^h}$. This will be summarized in the following lemma.
	
	\begin{Lemma}\label{frac_calc}
		Let $\alpha\not\in\Fq$, $k\in\Fq$, and $p\in\Fq[x]$ be a monic polynomial of degree $h>1$ for which $p(\alpha)=0$. Then there exist coefficients $a_{1,k},\dots,a_{h-1,k} \in \F_q$ such that 
		\[
		(\alpha-k)\left(\alpha^{h-1}+a_{1,k}\alpha^{h-2}+\dots+a_{h-2,k}\alpha+a_{h-1,k}\right)=-p(k).
		\] 
		In particular, if $m$ is the minimal polynomial of $\alpha$, and $k\in\Fq$, then there exist unique coefficients $c_{1,k},\dots,c_{h-1,k} \in \F_q$ such that 
		\[
		\frac1{\alpha+k}=-\frac1{m(-k)}\left(\alpha^{h-1}+c_{1,k}\alpha^{h-2}+\dots+c_{h-2,k}\alpha+c_{h-1,k}\right).
		\]
	\end{Lemma}
	\begin{proof}
		Consider the polynomial $r_k(x)=p(x)-p(k)$. Since $r_k(k)=0$, and $r_k$ is a monic polynomial, there exists a monic polynomial $w_r(x)=x^{h-1}+a_{1,k}x^{h-2}+\dots+a_{h-1,k}$ of degree $h-1$ such that $r_k(x)=(x-k)w_k(x)$. Substituting $\alpha$ and using $p(\alpha)=0$, we obtain our first equation.
		
		For the second equation, we apply the first one with $p=m$. We need $\alpha+k\neq 0$ and $m(-k)\not=0$ for all $k\in\Fq$, which holds since $\alpha\not\in\Fq$. Substituting $k\mapsto -k$ into the first equation and dividing both sides by $-(\alpha+k)\cdot m(-k)$, we get our second equation (with $c_{i,k}:=a_{i,-k}$). The coefficients $c_{i,k}$ are unique since $1,\alpha,\dots,\alpha^{h-1}$ are $\F_q$-linearly independent.
	\end{proof}
	
		
		\begin{Prop}
			Let $\alpha\in\Fqh\setminus\Fq$ with $\Fqh=\Fq(\alpha)$, $\beta:=1$ and define $f$ as in \ref{hgeq4_f_def} for some $s,t\in\Fq$. Then, there exist at least $q-1$ pairs $(s,t)$ such that the conditions of Theorem \ref{3fold_construction} hold for $(\alpha,1,f)$.
		\end{Prop}
		\begin{proof}
			Substituting $\beta=1$ into the first condition of Theorem \ref{3fold_construction}, we obtain $f(\alpha)\neq 1$, which is true, because we set $f(\alpha)=0$.
			
			Substituting into the second condition of Theorem \ref{3fold_construction}, we obtain
			\begin{equation}
				\label{fin}
				k\not=f(\alpha^2)+kf\left(\frac{1}{\alpha+k}\right)
			\end{equation}
			for all $k\in\Fq$. For $k=0$, this condition is $f(\alpha^2)\not=0$, meaning $s\not=0$. We now consider all $q^2-q$ pairs $(s,t)\in\Fq\setminus\{0\}\times\Fq$, and count how many pairs fail to satisfy the second condition for some $k\in\Fq\setminus\{0\}$.
			By Lemma \ref{frac_calc}, if $m$ is the minimal polynomial of $\alpha$, then
			\[
			\frac1{\alpha+k}=-\frac1{m(-k)}\left(\alpha^{h-1}+c_{1,k}\alpha^{h-2}+\dots+c_{h-2,k}\alpha+c_{h-1,k}\right)
			\] for some unique $c_{i,k}\in\Fq$, meaning 
			\[f\left(\frac1{\alpha+k}\right)=\frac{c_{h-1,k}+c_{h-3,k}s+t}{-m(-k)},
			\] from our definition of $f$. Substituting into \eqref{fin} and multiplying by $\frac{m(-k)}k\neq 0$, we obtain that a pair $(s,t)$ fails the second condition for a given $k\in\Fq\setminus\{0\}$ if only if \[
			-t-c_{h-3,k}s-c_{h-1,k}+\frac{m(-k)}ks=m(-k)\Longleftrightarrow t=-m(-k)-c_{h-1,k}-c_{h-3,k}s+\frac{m(-k)}ks,
			\]
			which is a linear equation in $t$, therefore for every $s,k\in\Fq\setminus\{0\}$, there is exactly one $t$ for which $(s,t)$ fails to satisfy the second condition. 
			
			It follows that there are at least $(q^2-q)-(q-1)^2=q-1$ pairs $(s,t)$ that satisfy both conditions of Theorem \ref{3fold_construction}.
		\end{proof}

		\subsection{The case of \texorpdfstring{$h=3$}{h=3}}
		
		In this case, we use a slightly different construction, because when defining $f$ we have only three degrees of freedom instead of at least four.
		
		Again, put $\beta:=1$. We will choose an appropriate $\alpha\in\F_{q^3}\setminus\Fq$ and show the existence of an $f$ such that the triple $(\alpha,1,f)$ satisfies the conditions of Theorem \ref{3fold_construction}. As it turns out, the following lemma will be very useful for our construction:\begin{Lemma}\label{lambda}
			The number of triples $(\lambda_0,\lambda_1,\lambda_2)\in\Fq^3$, where the polynomial $p(x)=x^3-\lambda_2x^2-\lambda_1x-\lambda_0$ is irreducible over $\Fq$, and the polynomial $r(x)=x^3+\lambda_2x^2-(\lambda_1-1)x+\lambda_0$ is reducible over $\Fq$, is at least \[
			\frac{q(q-1)^2}{9}-\frac{2(q-1)}3.
			\]
		\end{Lemma}
		\begin{proof}
			Fix $r(x)=x^3+\lambda_2x^2-(\lambda_1-1)x+\lambda_0$, and assume that it is reducible. Since it has degree $3$, it follows that $r$ has a root in $\Fq$. Let us fix $k\in \F_q$ as a root. If $k=0$, then $\lambda_0=0$, so the polynomial $p(x)=x^3-\lambda_2x^2-\lambda_1x-\lambda_0$ is also not irreducible, so we will assume $k\not=0$. 
			If $r(k)=0$, then 
			\begin{equation}
				\label{r(k)=0}
				k^2\lambda_2-k\lambda_1+\lambda_0=-k^3-k,
			\end{equation}
			so there are exactly $q^2$ choices for $(\lambda_0,\lambda_1,\lambda_2)$ such that $r(k)=0$. We will give a lower bound on the number of triples $(\lambda_0,\lambda_1,\lambda_2)$ for which $p(x)$ is irreducible and $r(k)=0$. This will be achieved by giving an upper bound on all remaining cases, namely those for which $r(k)=0$ and $p(x)$ is reducible. We will make use of Vieta's formulas: If $p(x)=x^3-\lambda_2x^2-\lambda_1x-\lambda_0=(x-r_1)(x-r_2)(x-r_3)$, then $\lambda_2=r_1+r_2+r_3$, $-\lambda_1=r_1r_2+r_2r_3+r_1r_3$, $\lambda_0=r_1r_2r_3$. 
			\begin{enumerate}
				\item $p(x)$ has one root of multiplicity $3$, meaning \[
				p(x)=(x-r_1)^3\text{, }r_1\in\Fq.
				\]After substituting Vieta's formulas into \eqref{r(k)=0}, we obtain $r_1^3-3kr_1^2+3r_1k^2=-k^3-k$. This is a cubic equation in $r_1$, meaning there are at most $3$ solutions, so at most $3$ triples $(\lambda_0,\lambda_1,\lambda_2)$ for which $r(k)=0$ and $p(x)$ has a $3$-fold root in $\F_q$.
				
				\item $p(x)$ has two distinct roots, one with multiplicity $2$, meaning \[
				p(x)=(x-r_1)(x-r_2)^2\text{, }r_1,r_2\in\Fq,\, r_1\neq r_2. 
				\]
				Similarly, substitution into \eqref{r(k)=0} gives $r_1(k^2+2kr_2+r_2^2)=-k^3-k-2k^2r_2-kr_2^2$, or equivalently,  $r_1(k+r_2)^2=-k(k+r_2)^2-k$. If $r_2=-k$, we get $k=0$, a contradiction. If $r_2\not=-k$, we have a linear equation in $r_1$, so for each $r_2\in\Fq\setminus\{-k\}$, we have exactly one solution for $r_1$. This means that there are at most $(q-1)$ triples $(\lambda_0,\lambda_1,\lambda_2)$, one for every $r_2\in\Fq\setminus\{-k\}$,  for which $r(k)=0$ and $p(x)$ has two distinct roots, one with multiplicity $2$. 
				\item $p(x)$ has three distinct roots, meaning \[
				p(x)=(x-r_1)(x-r_2)(x-r_3)\text{, }r_1,r_2,r_3\in\Fq \text{ are pairwise distinct.}
				\]
				Similarly to the previous case, substitution into \eqref{r(k)=0} gives us $r_3(k^2+kr_1+kr_2+r_1r_2)=-k(k^2+kr_1+kr_2+r_1r_2)-k$. If $k^2+kr_1+kr_2+r_1r_2=0$, we again get $k=0$, a contradiction. If $k^2+kr_1+kr_2+r_1r_2\not=0$, then we get a linear equation in $r_3$, thus, for any pair $r_1,r_2$ there is at most one $r_3$ such that this case holds. Since $r_1$ and $r_2$ are distinct, and their roles are symmetric, we can choose them in $q\choose2$ ways. In these $q\choose2$ cases, we counted every $\{r_1,r_2,r_3\}$ unordered triple 3 times, because we fixed $r_3$ to be the third, meaning, for this case there are at most $\frac13{q\choose2}=\frac{q^2-q}{6}$ triples $(\lambda_0,\lambda_1,\lambda_2)$ for which $r(k)=0$ and $p(x)$ has three pairwise distinct roots. 
				\item $p(x)$ has one root in $\Fq$ and two conjugate roots in $\F_{q^2}$, meaning \[
				p(x)=(x-r_1)(x-r_2)(x-r_2^q), \, r_1\in\Fq,\, r_2\in\F_{q^2}\setminus\Fq.
				\]
				As in the previous case, we can see that for any choice of $r_2$ (which determines $r_2^q$) there is at most 1 suitable choice for $r_1$. Since $r_2\in\F_{q^2}\setminus\Fq$, we can fix either $r_2$ or $r_2^q$, meaning there are $\frac{q^2-q}{2}$ choices. This means that in this case there are at most $\frac{q^2-q}2$ possible triples  $(\lambda_0,\lambda_1,\lambda_2)$ for which $r(k)=0$ and $p(x)$ has an irreducible factor of degree $2$.
			\end{enumerate}
			Using these upper bounds, we can get a lower bound on the number of triples where $p(x)$ is irreducible and $k$ is a root of $r(x)$:\[
			\left\lvert\{p(x)\text{ is irreducible over }\Fq : r(k)=0\}\right\rvert\geq q^2-3-(q-1)-\frac{q^2-q}6-\frac{q^2-q}2=\frac{q(q-1)}3-2.
			\]
			Summing over all $k\in\Fq\setminus\{0\}$ gives us:\[
			\frac{q(q-1)^2}{3}-2(q-1)\leq\sum_{k\in\F_q^*}|\{(\lambda_0,\lambda_1,\lambda_2) : r(k)=0,\, p \text{ is irreducible}\}|\leq\]
			\[3|\{(\lambda_0,\lambda_1,\lambda_2) : p\text{ is irreducible},\, r\text{ is reducible}\}|,
			\]
			where we obtained the last inequality because $r$ can have at most $3$ roots in $\F_q$, so each pair $(p,r)$, with $p$ irreducible and $r$  reducible, is counted at most $3$ times. Dividing by $3$ gives the inequality of the statement.
		\end{proof}
		
		Now we are ready to prove our main claim.
		
		\begin{Prop}
			Let $q>10$ be a prime power. There exists an $\alpha\in\F_{q^3}\setminus\Fq$ and an $f \colon \F_{q^3}\rightarrow\Fq$ linear functional such that the triple $(\alpha,1,f)$ satisfies the conditions of Theorem \ref{3fold_construction}.
		\end{Prop}
		\begin{proof}
			First, take an arbitrary $\alpha\in\F_{q^3}\setminus\Fq$, we will specify the precise conditions on $\alpha$ later. Define $f$ on the basis $1,\alpha,\alpha^2$ by $f(1)=1$, $f(\alpha)=0$, and  $f(\alpha^2)=l$, where $l\in\Fq$ will be  chosen later. The conditions that $f$ needs to satisfy (after substituting $\beta=1$ and the known values of $f$) are:\begin{enumerate}
				\item $f(\alpha)\not=1$ (the first condition of Theorem \ref{3fold_construction}),
				\item $0\not=l$ (the second condition of Theorem \ref{3fold_construction} with $k=0$),
				\item $k\neq l+kf\left(\frac{1}{\alpha+k}\right)$ for all $k\in \F_q^*$, or equivalently, 
				\begin{equation}
					\label{fin2}
					f\left(\frac{1}{\alpha+k}\right)\not=1-\frac{l}k,
				\end{equation} 
			\end{enumerate}
			for all $k\in \F_q^*$ (the second condition of Theorem \ref{3fold_construction} with $k\in \F_q^*$). 
			
			\medskip
			
			For a fixed $k\in\Fq^*$, we can write \begin{equation}
				\label{eqx}
				\frac1{\alpha+k}=c_k+a_k\alpha+b_k\alpha^2,
			\end{equation} with $a_k,b_k,c_k\in\Fq$, since $\{1,\alpha,\alpha^2\}$ is an $\F_q$-basis of $\F_{q^3}$. Applying $f$, we obtain:
			\begin{equation}
				\label{fin3}
				f\left(\frac1{\alpha+k}\right)=c_k+lb_k,
			\end{equation}
			so condition \eqref{fin2} reads as  $c_k+lb_k\not=1-\frac{l}k$, or equivalently, $(b_k+\frac1k)l\not=1-c_k$. This means that a triple $(\alpha,1,f)$ cannot  satisfy the conditions of Theorem \ref{3fold_construction} if there exists $k\in \F_q^*$ such that $b_k=-\frac1k$ and $c_k=1$. So, our first condition on $\alpha$ is that for all $a\in\Fq$, $k\in\Fq^*$:
			\[
			\frac1{\alpha+k} \neq 1+a\alpha-\frac1k\alpha^2,
			\] 
			cf. \eqref{eqx}. For $a\in \F_q$ and $k\in \F_q^*$ this is a cubic equation in $\alpha$ (after multiplying by $\alpha+k$), thus it has at most $3$ roots, so the number of choices for $\alpha$ that we have to exclude is bounded above by $q+3q(q-1)$ (the roots of these $q(q-1)$ polynomials and the elements of $\Fq$). 
			
			Choosing $\alpha$ like this means that if for some $k\in \F_q^*$ it holds that $b_k=-\frac1k$, then, for this $k$, all $l\in \F_q$ satisfy \eqref{fin2}.
			If $b_k\not=-\frac1k$, then $(b_k+\frac1k)l=1-c_k$ is linear in $l$, thus, for this $k$, there is exactly one $l\in\Fq$ which doesn't satisfy \eqref{fin2}. If we can choose $\alpha$ such that for some $k\in\Fq^*$ we have $b_k=-\frac1k$, then, out of the $q-1$ possible nonzero values of $l$, there are at most $q-2$ values which, for some $k$, violate \eqref{fin2}. So, in this case we can choose $l$ such that $(\alpha,1,f)$ satisfies the conditions of Theorem \ref{3fold_construction}.
			
			\medskip
			
			We now uncover what $b_k=-\frac1k$ actually means. It is equivalent to: 
			\[
			\frac1{\alpha+k}=c_k+a_k\alpha-\frac1k\alpha^2\Leftrightarrow \alpha^3-k(a_k-1)\alpha^2-k(a_kk+c_k)\alpha-k(c_kk-1)=0.
			\]
			Thus $\alpha$ is a root of the monic cubic $x^3-k(a_k-1)x^2-k(a_kk+c_k)x-k(c_kk-1) \in \F_q[x]$. If we let the minimal polynomial of $\alpha$ be $p(x)=x^3-\lambda_2x^2-\lambda_1x-\lambda_0$, then $\lambda_2=k(a_k-1)$, $\lambda_1=k(a_kk+c_k)$ and $\lambda_0=k(c_kk-1)$. Eliminating $a_k$ and $c_k$ and solving for $k$, we obtain 
			\[
			k^3+\lambda_2k^2-(\lambda_1-1)k+\lambda_0=0,
			\]
			which is equivalent to the polynomial $r(x)=x^3+\lambda_2x^2-(\lambda_1-1)x+\lambda_0$ having a root in $\Fq$, i.e. being reducible.
			So, to have $b_k=-\frac1k$ for some $k\in \F_q^*$, we have to find a triple $(\lambda_0,\lambda_1,\lambda_2)\in \F_q^3$ such that $p(x)=x^3-\lambda_2x^2-\lambda_1x-\lambda_0$ is irreducible, and $r(x)=x^3+\lambda_2x^2-(\lambda_1-1)x+\lambda_0$ is reducible. By Lemma \ref{lambda}, there exist at least $\frac{q(q-1)^2}{9}-\frac{2(q-1)}3$ such triples. We can choose $\alpha$ to be any root of any such $p(x)$. Since these polynomials are irreducible and pairwise distinct, each of them has $3$ conjugate roots, and all of these roots are distinct, so there are at least $\frac{q(q-1)^2}{3}-{2(q-1)}$ choices for  $\alpha$.
			
			\medskip
			
			Note that in \eqref{fin3} we already had some conditions on $\alpha$. If $\frac{q(q-1)^2}{3}-2(q-1)>q+3q(q-1)$, then we can certainly choose an $\alpha$ for which there exists a suitable $l$, and hence $f$, such that $(\alpha,1,f)$ satisfies the conditions of Theorem \ref{3fold_construction}. This inequality holds for $q>10$, which proves our claim. 
		\end{proof}
		
		\begin{Rem}
			The prime powers below $11$ are $2,3,4,5,7,8,9$. In these cases, one can find a suitable $\alpha$ such that the conditions hold with $l=1$, using a simple brute-force computer program.
		\end{Rem}
		
		\begin{Cor}\label{3fold_main}
			For all prime powers $q$, there exists a $3$-fold blocking set of size $3(q^h+q^{h-1}+1)$ in $\FGq$, $h\geq 2$.
		\end{Cor}
		
		Theorem \ref{3fold_qh} follows from 
		Corollaries \ref{2fold_main} and \ref{3fold_main}. 
		
		\section*{Acknowledgement}
		
		All authors would like to thank the Budapest REU 2025 Math Program for undergraduate research.

		\vspace{1cm}
		
		\noindent Bence Csajb\'{o}k,\\
		Department of Computer Science\\
		ELTE E\"otv\"os Lor\'and University\\
		H-1117 Budapest, P\'azm\'any P.\ stny.\ 1/C, Hungary.\\ E-mail: {\texttt{bence.csajbok@ttk.elte.hu}}
		
		\vspace{1cm}
		
		\noindent M\'at\'e R\'obert Kepes, Eszter Melinda Robin,\\
		ELTE E\"otv\"os Lor\'and University\\
		H-1117 Budapest, P\'azm\'any P.\ stny.\ 1/C, Hungary.\\ E-mail: {\texttt{kepes.mate.robert@gmail.com, robin.eszti@gmail.com}}
		
		\vspace{1cm}
		
		\noindent Bence S\'ogor,\\
		KU Leuven\\
		Naamsestraat 22, 3000 Leuven, Belgium.\\
		\&\\
		Babe\c{s}-Bolyai University\\
		Str. Mihail Kog\u{a}lniceanu 1, 400084 Cluj-Napoca, Romania.\\ E-mail: {\texttt{sogorbence2002@gmail.com}}
		
		\vspace{1cm}
		\noindent Sherry Wang,\\
		Reed College \\
		3203 SE Woodstock Blvd, Portland, OR, USA.\\
		E-mail: \texttt{swang@reed.edu}
		
		\vspace{1cm}
		\noindent Elias Williams,\\
		Lewis \& Clark College \\
		615 S Palatine Hill Road, Portland, OR, USA.\\
		E-mail: \texttt{roark@lclark.edu}


\begin{thebibliography}{99} 
			
			
			
			\bibitem{dbs1}
			{\sc G. Bacs\'o, T. H\'eger, and T. Sz\H onyi:} The 2-blocking number and the upper chromatic number of $\mathrm{PG}(2,q)$, J.\ Combin.\ Des.\  {\bf 21}(12) (2013), 585--602.
			
			\bibitem{Ball0}
			{\sc S. Ball:} Multiple blocking sets and arcs in finite planes, J.\ London Math.\
			Soc. {\bf 54} (1996), 581--593.
			
			\bibitem{BSP}
			{\sc R. D. Baker, J. M. Dover, G. L. Ebert, K. L. Wantz:}
			Baer subgeometry partitions, J. Geom. {\bf 67} (2000), 23--34.
			
			\bibitem{Blokhuis0}
			{\sc A. Blokhuis:}
			On the size of a blocking set in $\mathrm{PG}(2,p)$, Combinatorica {\bf 14} (1994), 111--114.
			
			
			\bibitem{Bsurvey}
			{\sc A. Blokhuis:} Blocking sets in Desarguesian planes, in Combinatorics, Paul Erd\H os is Eighty, Vol. 2, J\'anos Bolyai Mathematical Society, Budapest, (1994), 133--155.
			
			\bibitem{mbs1}
			{\sc A. Blokhuis, L. Storme, and T. Sz\H onyi:} Lacunary polynomials, multiple blocking sets and Baer subplanes, J.\ London Math.\ Soc.\ (2), {\bf 60}(2) (1999), 321--332.
			
			\bibitem{arcs}
			{\sc M. Braun, A. Kohnert, and A. Wassermann:} Construction of \((n, r)\)-arcs in \(\mathrm{PG} (2, q)\), Innov.\ Incidence Geom. {\bf 1} (2005), 133--141.
			
			\bibitem{Bruen}
			{\sc A. Bruen:} Baer subplanes and blocking sets, Bull.\ Amer.\ Math.\ Soc. {\bf 76} (1970), 342--344.
			
			\bibitem{Bruen1}
			{\sc A. Bruen:} Blocking sets in finite projective planes, Siam J.\ Appl.\ Math.  {\bf 21} (1971), 380--392.
			
			\bibitem{Bruen0}
			{\sc A. A. Bruen:} Polynomial Multiplicities over Finite Fields
			and Intersection Sets, J. Combin.\ Theory Ser. A {\bf 60} (1992), 19--33.
			
			
			\bibitem{Csajb}
			{\sc B. Csajb\'ok:} Unpublished manuscript, 2025.
			
			\bibitem{3q-1double}
			{\sc B. Csajb\'ok, T. H\'eger:} Double blocking sets of size \(3q-1\) in \(\mathrm{PG} (2, q)\), European J.\ Combin. {\bf 78}  (2019), 73--89.
			
			\bibitem{Daskalov}
			{\sc R. Daskalov:} New good large $(n,r)$-arcs in $\mathrm{PG}(2,29)$ and $\mathrm{PG}(2,31)$, J.\ Algebra Comb.\ Discrete Struct.\ Appl. {\bf 11} (2024), 93--104.
			
			\bibitem{dbvdv}
			{\sc J. De Beule, G. Van de Voorde:} The minimum size of a linear set, J. Combin.\ Theory Ser.\ A {\bf 164} (2019), 109--124.
			
			\bibitem{dbs2}
			{\sc J. De Beule, T. H\'eger, T. Sz\H onyi, G. Van de Voorde:}
			Blocking and double blocking sets in finite planes, Electron.\ J.\ Combin. (2016), \#P2.5
			
			
			\bibitem{JWP} 
			{\sc J.W.P. Hirschfeld:} Projective Geometries over Finite Fields, second ed., Clarendon Press, Oxford, 1979, 1998.
			
			\bibitem{lowerboundprime}
			{\sc Gy. Kiss, T. Szonyi:} Finite geometries. Chapman and Hall/CRC (2019), 172--173.
			
			\bibitem{Lunardon}
			{\sc G. Lunardon:} Normal Spreads, Geom.\ Dedicata {\bf 75} (1999), 245--261.
			
			\bibitem{unpub2}
			{\sc O. Polverino, L. Storme:}
			Unpublished manuscript, 2000.
			
			\bibitem{PoSt}
			{\sc O. Polverino, L. Storme:}
			Small Minimal Blocking Sets in $\mathrm{PG}(2,q^3)$, European J.\ Combin. (2002) {\bf 23}, 83--92.
			
			\bibitem{linconj}
			{\sc P. Sziklai:} On small blocking sets and their linearity, 
			J.\ Combin. Theory Ser.\ A {\bf 115}(7) (2008), 1167--1182.
			
			\bibitem{1modp}
			{\sc T. Sz\H onyi:} Blocking sets in Desarguesian affine and projective planes, Finite Fields Appl.\ {\bf 3} (1997), 187--202.
			
			
			
			
		\end{thebibliography}
	\end{document}